\documentclass[12pt]{amsart}
\usepackage{graphicx} 
\usepackage[utf8]{inputenc}
\usepackage{amsmath,amssymb,amsfonts,amsthm}
\usepackage{hyperref}

\newtheorem{thm}{Theorem}[section]
\newtheorem{lem}[thm]{Lemma}

\newtheorem{defin}[thm]{Definition}
\newtheorem{cor}[thm]{Corollary}
\newtheorem{que}[thm]{Open Question}

\begin{document}

\title{Models of an Abstract Elementary Class as a Generalized Polish Space}

\author{Georgios Marangelis}
\email{geormara@math.auth.gr}
\address{Department of Mathematics, Aristotle University of Thessaloniki, Greece}


\thanks{This paper is part of the author's Ph.D. Thesis at the Department of Mathematics of Aristotle University in Thessaloniki, Greece, under the guidance of Dr. Ioannis Souldatos.}

\begin{abstract}
    In first order logic, it is known that you can define a topology so that the countable models of some theory $T$ form a Polish Space (i.e. completely metrizable second countable space).

    In this paper we use the Baldwin- Boney Relational Presentation Theorem (from \cite{PrTh}; cf. \ref{prth}) to generalize this result to the models of an Abstract Elementary Class (AEC). More specifically, we define a topology on the models of an AEC of size $\lambda \geq \kappa$, where $\kappa$ is the L$\ddot{o}$wenheim-Skolem number and $\lambda$ has to satisfy a set-theoretic assumption (see Section \ref{sectiontopologyAEC}) and prove that these models form a Generalized Polish Space (i.e. a generalization of Polish Spaces i.e. completely $G$-metrizable space with weight $\leq \kappa$). 

\end{abstract}

\maketitle

\section{Introduction}
\subsection{Countable Models of First-Order Theories}
~

First we survey some results about the topological space of countable models of a first-order theory $T$.

\begin{defin}\label{Polish}
    A topological space $X$ is called $\mathbf{Polish}$ if it is completely metrizable and second countable. 
\end{defin}

In first order logic, if $T$ is a theory, we can define a topology on the countable models of $T$ and prove that it is a Polish Space. In fact, we can define two different topologies on such space and prove that both make it a Polish Space. 

\begin{defin}\label{topology1}
Let $\mathcal{L}$ be some countable language and let $Mod_{\omega}(\mathcal{L})$ be the space of all $\mathcal{L}$-structures on $\omega$ (since all of our models are countable, we can think that their universe is always $\omega$) equipped with the $\mathbf{topology}$ whose basic open sets are of the following form:

\begin{center}
    $\{ M | M \models \phi(n_{1}, n_{2}, ..., n_{k}) \}$
\end{center}

for $k, n_{1}, ..., n_{k} \in \omega$ and $\phi$ a quantifier-free first order formula.
\end{defin}

\begin{defin}\label{topology2}
    One can define a slightly different topology from the one above. More specifically, we can define the basic open sets to be of the form:

\begin{center}
    $\{ M | M \models \phi(n_{1}, n_{2}, ..., n_{k}) \}$
\end{center}

for $k, n_{1}, ..., n_{k} \in \omega$ and $\phi$ a first order formula.

\end{defin}

We include the next Theorem without proof. A full proof can be found in \cite{Hjorth}.

\begin{thm}\label{firstorderPolish}
    The space $Mod_{\omega}(\mathcal{L})$ equipped either with the topology defined in Definition \ref{topology1}, or with the topology defined in Definition \ref{topology2} is a Polish Space.
\end{thm}

We can also define a similar topology for $\mathcal{L}_{\omega_{1}, \omega}$-formulas, as follows:

\begin{defin}\label{topology3}
    For $\mathcal{F} \subseteq \mathcal{L}_{\omega_{1}, \omega}$ a fragment, we define $T_{\mathcal{F}}$ to be the topology on $Mod_{\omega}(\mathcal{L})$ generated by open sets of the form:

    \begin{center}
        $\{ M \in Mod_{\omega}(\mathcal{L}) | M \models \phi(n_{1}, n_{2}, ..., n_{k})\}$,
    \end{center}

    where $k, n_{1}, ..., n_{k} \in \omega$ and $\phi(\bar{x}) \in \mathcal{F}$.
\end{defin}

\begin{thm}\label{InfinitaryPolish}
    If $\mathcal{F}$ is a countable fragment of $\mathcal{L}_{\omega_{1}, \omega}$, then $(Mod_{\omega}(\mathcal{L}), T_{\mathcal{F}})$ is a Polish Space.
\end{thm}

We will see some Theorems relevant to the topologies we defined, without referring to their proofs. In \cite{InfinitaryMarker} is presented a proof of Morley's Theorem using the topology we defined with the quantifier-free formulas.

\begin{thm}\label{Morley}
    (Morley) Let $\phi$ be an $\mathcal{L}_{\omega_{1}, \omega}$ sentence. The number of non-isomorphic countable models of $\phi$ is either at most $\aleph_{1}$ or exactly $2^{\aleph_{0}}$.
\end{thm}

\begin{defin}\label{meager}

\begin{enumerate}
    \item A subset of a topological space $X$ is $\mathbf{nowhere}$ $\mathbf{dense}$, if its intersection with any non-empty open subset of $X$ is not dense.
    \item A subset of $X$ is called $\mathbf{meager}$ if it can be written as the countable union of closed nowhere dense sets.
    \item The complement of a meager set is called $\mathbf{comeager}$.
\end{enumerate}

\end{defin}

In \cite{Hjorth} we find the next results (which also hold in the first order case):

\begin{lem}\label{Hjorth1}
    Let $\mathcal{F} \subseteq \mathcal{L}_{\omega_{1}, \omega}$ be a countable fragment. Let $\Sigma(x_{1}, x_{2}, ..., x_{n})$ be a complete non-principal type over $\mathcal{F}$ and let $k_{1}, k_{2}, ..., k_{n} \in \omega$. Let $T \subseteq \mathcal{F}$ be a complete theory. Then the set:

    \begin{center}
        $\{M \in Mod_{\omega}(T) | M \models \Sigma(k_{1}, ..., k_{n}) \}$
    \end{center}

    is closed and nowhere dense (for the topology $T_{\mathcal{F}}$).
\end{lem}

\begin{cor}\label{Hjorth2}
    (Omitting Types Theorem) Let $\mathcal{F}, T$ and $\Sigma(\bar{x})$ as in the above Lemma. Then the set

    \begin{center}
        $\{ M \in Mod_{\omega}(T) | M$ omits $\Sigma(\bar{x})\}$ 
    \end{center}

    is comeager in $(Mod_{\omega}(T), T_{\mathcal{F}})$.
\end{cor}

\begin{thm}\label{Hjorth3}
    Let $\mathcal{F}$ be a countable fragment of $\mathcal{L}_{\omega_{1}, \omega}$, $T \subseteq \mathcal{F}$ a complete theory and $M_{0} \in Mod_{\omega}(T)$. The set of $N \in (Mod_{\omega}(T), T_{\mathcal{F}})$ which are isomorphic to $M_{0}$, is comeager if and only if $M_{0}$ is an atomic model.
\end{thm}

\subsection{Abstract Elementary Classes}
S. Shelah introduced the notion of Abstract Elementary Classes (AEC) in the 1980's, in order to generalize some of his own results in first order logic. AEC's are a non-syntactic generalization of Elementary Classes in first order logic. The reader can consult \cite{Categoricity} for more details and proofs of the results presented here.

Shelah was also the first one who proved a Theorem that allows us to replace the entirely semantic description of AEC's by a syntactic one. This Theorem is known as "Shelah's Presentation Theorem".

\begin{thm}\label{Shelah'sPresentationTheorem}
    (Shelah's Presentation Theorem) Let $\tau$ be a vocabulary and $\mathbf{K}$ be an AEC in $\tau$ ($|\tau| \leq LS(\mathbf{K})$). Then there is a vocabulary $\tau' \supseteq \tau$ with cardinality $|\tau|$, a first order theory $T'$ and a set $\Gamma$ of at most $2^{LS(\mathbf{K})}$ partial types, such that:
    \begin{center}
        $\mathbf{K} = \{ M'|_{\tau} : M' \models T'$ and $M'$ omits $\Gamma \}$.
    \end{center}
    Moreover, the $\prec_{\mathbf{K}}$ relation satisfies the following:
    \begin{itemize}
        \item If $M'$ is a $\tau'$-substructure of $N'$, where $M', N' \models T'$ and omit $\Gamma$, then $M'|_{\tau} \prec_{\mathbf{K}} N'|_{\tau}$.
        \item If $M \prec_{\mathbf{K}} N$ there is an expansion of $N$ to a $\tau'$-structure such that $M$ is the universe of a $\tau'$-substructure of $N'$.
    \end{itemize}
\end{thm}

From then there have been proved many different "Presentation Theorems" for AEC's. All of them are constructed for a different purpose. In this paper, we use the Baldwin and Boney ``Relational Presentation Theorem" (see Section \ref{sectionPrTh}). Our purpose is to obtain the necessary syntactic tools and define a topology on the models of an AEC, of size $\lambda$, for $\lambda \geq LS(\mathbf{K})$. 

For our purposes, the Baldwin-Boney's  ``Relational Presentation Theorem" has certain advantages than Shelah's original theorem. First, every model in the AEC has a unique expansion in the expanded vocabulary and second, there is no set of partial types to omit. 

From the variety of Presentation Theorems, noteworthy is the theorem in \cite{Grossberg}, where Samson Leung improves the results of the ``Relational Presentation Theorem". One interesting part of Leung's theorem is that it keeps the original language. One can ask if this theorem can be used in the present paper instead of the relational presentation theorem, but this is something we have not examined.

\subsection{Generalized Descriptive Set Theory}

An important observation is that in this paper we will not necessarily define a topology on a class of countable models, but rather on a class of uncountable models. In the first-order case, we work with subspaces of $\omega^{\omega}$. In this paper, we work with subspaces of $\kappa^{\kappa}$, where $\kappa$ is some regular uncountable cardinal.  (see Sections \ref{sectiontopologyAEC}, \ref{sectionGentopologyAEC}). Since $\kappa^\kappa$ can not be a Polish Space, we work a proper generalization of this notion called $G$-Polish Space (see Section \ref{sectionGPS}).
The study of $G$-Polish spaces falls into Generalized Descriptive Set Theory and has been a very active field during the last decade, since it has connections with a lot of other fields in Mathematics. 

There are two ``natural" ways to generalize the notion of Polish Spaces, $\mathbf{strong}$ $\mathbf{\kappa}$-$\mathbf{Choquet}$ $\mathbf{spaces}$ (see \cite{GPS})  and $\mathbf{G}$-$\mathbf{Polish}$ $\mathbf{spaces}$ (see \cite{GPS} or Section \ref{sectionGPS}). The definitions are not equivalent in general, but in many cases we can equivalently use either of them. In our paper we use the definition of $G$-Polish Space, but we could have used  strong $\kappa$-Choquet spaces too. One can consult \cite{GPS,3GPS} and \cite{Baldwin} for more results in this area.

In this paper, we connect for the first time AEC's with Generalized Descriptive Set Theory. More specifically, we use the Baldwin-Boney Relational Presentation Theorem in order to obtain the necessary syntactic tools and define a topology on the models of size $LS(\mathbf{K})$ of an AEC $\mathbf{K}$. Then, working similarly to \cite{Hjorth}, we prove that these models form a Generalized Polish Space. Finally, in the last Section, we generalize this result to models of an AEC of size equal to $\lambda$, where $\lambda \geq LS(\mathbf{K}) = \kappa$, under the assumption that for $\mu= (\lambda + \kappa^{+})^{\kappa}$ it holds that $\mu^{< \mu} = \mu$.

One natural question in the context of our paper is how many of the results from first order logic generalize in this new setting. For instance:

\begin{que}
  Can we generalize the notion of atomic models to fit the framework of AEC's in such a way that the analogue of Theorem \ref{Hjorth3} for atomic models of size equal to the L$\ddot{o}$wenheim-Skolem number holds true?
\end{que}

In Section \ref{sectionPrTh} we present the Baldwin- Boney Relational Presentation Theorem and in Section \ref{sectionGPS} we define Generalized Polish Spaces and present some initial results for this notion. In Section \ref{sectiontopologyAEC} we prove the main result of the paper. More specifically, we define a topology on the models of an AEC of size equal to $\kappa = LS(\mathbf{K})$ (under the assumption that $(2^{\kappa})^{< 2^{\kappa}} = 2^{\kappa}$), and prove that they form a Generalized Polish Space. Finally, in Section \ref{sectionGentopologyAEC}, we generalize the last result for $\lambda > LS(\mathbf{K})$ (here we need the set theoretic assumption that for  $\mu=(\lambda + \kappa^{+})^{\kappa}$, $\mu^{<\mu}=\mu$).

\section{The relational presentation theorem}\label{sectionPrTh}

In this section we present the Relational Presentation Theorem for AEC's, following J. Baldwin and W. Boney. The reader who is familiar with \cite{PrTh} should skip this section.

\begin{defin}\label{K_m}
    Let $\mathbf{K}$ be an AEC and $\mu$ a cardinal. Then, $\mathbf{K}_{\mu}$ is the class of all models in $\mathbf{K}$ of size $\mu$.
\end{defin}

First, we fix some notation. Let $\mathbf{K}$ be an AEC in a vocabulary $\tau$ and let $\kappa$ denote the L$\ddot{o}$wenheim-Skolem number $LS(\mathbf{K})$.  We assume that $\mathbf{K}$ contains no models of size $< LS(\mathbf{K})$. The same arguments given here could also be given for $\kappa > LS(\mathbf{K})$.

We fix some compatible enumerations for models $M \in \mathbf{K}_{\kappa}$. Compatible enumerations means that each M has an enumeration of its universe, denoted $\mathbf{m}^{M} = (m_{i}^{M} : i < \kappa)$, and, if $M \cong N$, there is some fixed isomorphism $f_{M, M'} : M \cong M'$ such that $f_{M, M'}(m_{i}^{M}) = m_{i}^{M'}$ and if $M \cong M' \cong M''$, then $f_{M, M''} = f_{M', M''} \circ f_{M, M'}$.

For each isomorphism type $[M]_{\cong}$ and $[M \prec_{\mathbf{K}} N]_{\cong}$ with M, N $\in \mathbf{K}_{\kappa}$, we add to $\tau$ new predicates $R_{[M]}(\mathbf{x})$ and $R_{[M \prec_{\mathbf{K}} N]}(\mathbf{x}, \mathbf{y})$ which are $\kappa$-ary and $(\kappa \times 2)$-ary respectively and we form $\tau^{*} \supseteq \tau$.

Next, we define the presentation theory $T^{*}$. The purpose of this theory is to identify strong submodels of size $\kappa$ and strong submodel relations between these models via the new predicates $R_{[M]}$ and $R_{[M \prec_{\mathbf{K}} N]}$. This is done by expressing properties that connect the canonical enumerations with structures in $\mathbf{K}$ using the next axioms ($\mathbf{x}$ is a sequence of length at most $\kappa$).

\begin{center}
$R_{[M]}(\mathbf{x})$ holds iff $x_{i} \mapsto m_{i}^{M}$ is an isomorphism
\end{center}

\begin{center}
$R_{[M \prec_{\mathbf{K}} N]}(\mathbf{x}, \mathbf{y})$ holds iff $x_i \mapsto m_{i}^{M}$ and $y_{i} \mapsto m_{i}^{N}$ are isomorphisms  and $x_{i} = y_{j}$ iff $m_{i}^{M} = m_{j}^{N}$
\end{center}

Note that by the coherence of the isomorphisms, the choice of representative from $[M]_{\cong}$ doesn't matter. Also, we might have $M \cong M', N \cong N', M \prec_{\mathbf{K}} N$ and $M' \prec_{\mathbf{K}} N'$, but not $(M, N) \cong (M', N')$. In this case $R_{[M \prec_{\mathbf{K}} N]}$ and $R_{[M' \prec_{\mathbf{K}} N']}$ are different predicates.

We, now, write the axioms of $T^{*}$.

\begin{defin}\label{dT*}
    The $\mathbf{theory}$ $\mathbf{T^{*}}$ in $\mathcal{L}_{(I(\mathbf{K}, \kappa) + \kappa)^{+}, \kappa^{+}}(\tau^{*}) \subseteq \mathcal{L}_{(2^{\kappa})^{+}, \kappa^{+}}(\tau^{*})$ is the collection of the following schemata:

    \begin{enumerate}
        \item If $R_{[M]}(\mathbf{x})$ holds, then $x_{i} \mapsto m_{i}^{M}$ should be an isomorphism.\\
        If $\phi(z_1, ..., z_n)$ is an atomic or negated atomic $\tau$- formula that holds of $m_{i_{1}}^{M}, ..., m_{i_{n}}^{M}$, then include
        \begin{center}
            $\forall \mathbf{x} (R_{[M]}(\mathbf{x}) \to \phi(x_{i_{1}}, ..., x_{i_{n}}))$
        \end{center}
        \item If $R_{[M \prec_{\mathbf{K}} N]}(\mathbf{x}, \mathbf{y})$ holds, then $x_{i} \mapsto m_{i}^{M}$ and $y_{i} \mapsto m_{i}^{N}$ should be isomorphisms and the correct overlap should occur.\\
        If $M \prec_{\mathbf{K}} N$ and $i \mapsto j_{i}$ is the function such that $m_{i}^{M} = m_{j_{i}}^{N}$, then include 
        \begin{center}
            $\forall \mathbf{x}, \mathbf{y} (R_{[M \prec_{\mathbf{K}} N]}(\mathbf{x}, \mathbf{y}) \rightarrow (R_{[M]}(\mathbf{x}) \wedge R_{[N]}(\mathbf{y}) \wedge \bigwedge_{i < \kappa} x_{i} = y_{j_{i}}))$
        \end{center}
        \item Every $\kappa$-tuple is covered by a model.\\
        Include the following where lg$(\mathbf{x})$ = lg$(\mathbf{y})$ = $\kappa$
        \begin{center}
            $\forall \mathbf{x} \exists \mathbf{y} (\bigvee_{[M]_{\cong} \in \mathbf{K}_{\kappa}/ \cong} R_{[M]}(\mathbf{y}) \wedge \bigwedge_{i < \kappa} \bigvee_{j < \kappa} x_{i} = y_{j})$
        \end{center}
        
        \item If $R_{[N]}(\mathbf{x})$ holds and $M \prec_{\mathbf{K}} N$, then $R_{[M \prec_{\mathbf{K}} N]}(\mathbf{x}^{o}, \mathbf{x})$ should hold for the appropriate subtuple $\mathbf{x}^{o}$ of $\mathbf{x}$.\\
        If $M \prec_{\mathbf{K}} N$ and $\pi : \kappa \to \kappa$ is the unique map so $m_{i}^{M} = m_{\pi(i)}^{N}$, then denote $\mathbf{x}^{\pi}$ to be the subtuple of $\mathbf{x}$ such that $x_{i}^{\pi} = x_{\pi(i)}$ and include 
        \begin{center}
            $\forall \mathbf{x} (R_{[N]}(\mathbf{x}) \rightarrow R_{[M \prec_{\mathbf{K}} N]}(\mathbf{x}^{\pi}, \mathbf{x}))$
        \end{center}
        \item Coherence: If $M \subseteq N$ are both strong substructures of the whole model, then $M \prec_{\mathbf{K}} N$.\\
        If $M \prec_{\mathbf{K}} N$ and $m_{i}^{M} = m_{j_{i}}^{N}$, then include
        \begin{center}
            $\forall \mathbf{x}, \mathbf{y} (R_{[M]}(\mathbf{x}) \wedge R_{[N]}(\mathbf{y}) \wedge \bigwedge_{i < \kappa} x_{i} = y_{i_{j}} \rightarrow R_{[M \prec_{\mathbf{K}} N]}(\mathbf{x}, \mathbf{y}))$
        \end{center}
    \end{enumerate}
    \end{defin}

    \underline{Remark:} The converse of (1) of Definition \ref{dT*} is not true.

    \begin{thm}(Relational Presentation Theorem)\label{prth}
If $M^{*} \models T^{*}$, then we write just $M$ instead of $M^{*}|_{\tau}$.
\begin{enumerate}
    \item If $M^{*} \models T^{*}$ then $M^{*}|_{\tau} \in \mathbf{K}$. Further, for all $M_{0} \in \mathbf{K}_{\kappa}$, we have $M^{*} \models R_{[M_{0}]}(\mathbf{m})$ implies that $\mathbf{m}$ enumerates a strong submodel of M.
    \item Every $M \in \mathbf{K}$ has a unique expansion $M^{*}$ that models $T^{*}$.
    \item If $M \prec_{\mathbf{K}} N$, then $M^{*} \subseteq N^{*}$.
    \item If $M^{*} \subseteq N^{*}$ both model $T^{*}$, then $M \prec_{\mathbf{K}} N$.
    \item If $M \prec_{\mathbf{K}} N$ and $M^{*} \models T^{*}$ such that $M^{*}|_{\tau} = M$, then there is $N^{*} \models T^{*}$ such that $M^{*} \subseteq N^{*}$ and $N^{*}|_{\tau} = N$.
\end{enumerate}
    \end{thm}

    \section{Generalized Polish Spaces}\label{sectionGPS}
In this section we present the basic facts about Generalized Polish spaces. We follow Claudio Agostini, Luca Motto Ros and Philip Schlicht (cf. \cite{GPS}) in our presentation.
   
   More specifically, we study the space $\kappa^\kappa$ and its subspaces, where $\kappa$ is an uncountable regular cardinal. This generalizes the results about $\omega^\omega$ seen in  Descriptive Set Theory. 
   
   The two main spaces are: 
   
    \begin{enumerate}
        \item The \underline{Generalized Baire Space}
        \begin{center}
        $\kappa^{\kappa} = \{ x | x: \kappa \to \kappa \}$
        \end{center}
        of all sequences with values in $\kappa$ and length $\kappa$, equipped with the bounded topology $\tau_{b}$, i.e. the topology generated by the sets of the form 
        \begin{center}
            $N_{s} = \{ x \in \kappa^{\kappa} | s \subseteq x \}$
        \end{center}
        where $s \in \kappa^{< \kappa}$.
        \item The \underline{Generalized Cantor Space}
        \begin{center}
            $2^{\kappa} = \{ x | x: \kappa \to 2 \}$
        \end{center}
        which is a closed subset o $\kappa^{\kappa}$, equipped with the relative topology.
        
    \end{enumerate}

    \begin{defin}\label{weight}
    A topological space $X$ has $\mathbf{weight}$ $\kappa$, if there is a base for its topology of size $\kappa$.
    \end{defin}

    Since the classical Cantor and Baire Spaces are second countable, it is natural to require accordingly that $\kappa^{\kappa}$ and $2^{\kappa}$ have weight $\kappa$: this amounts to requiring that $\kappa^{< \kappa} = \kappa$, or equivalently that $\kappa$ is regular and $2^{< \kappa} = \kappa$. Thus such assumption is one of the basic conditions in the development of the theory of Generalized Polish Spaces.

Recall that a Polish Space is a completely metrizable second countable space. In order to generalize this we need some definitions. 

    Consider a totally ordered (Abelian) group
    \begin{center}
        $G = (G, +_{G}, 0_{G}, \leq_{G})$
    \end{center}
    with $\mathbf{degree}$ Deg(G) = $\kappa$, where Deg(G) is the coinitiality of $G^{+} = \{\epsilon \in G | 0_{G} <_{G} \epsilon \}$ of G. A $\mathbf{G}$-$\mathbf{metric}$ on a nonempty space X is a function $d: X^{2} \to G^{+} \cup \{0_{G}\}$ satisfying the usual rules of a distance function: for all $x, y, z \in X$
    \begin{itemize}
        \item $d(x, y) = 0_{G} \Leftrightarrow x = y$
        \item $d(x, y) = d(y, x)$
        \item $d(x, z) \leq_{G} d(x, y) +_{G} d(y, z)$
    \end{itemize}
    Every $G$-metric space $(X, d)$ is naturally equipped with the $d$-topology generated by its open balls
    \begin{center}
        $B_{d}(x, \epsilon) = \{y \in X | d(x, y) <_{G} \epsilon \}$
    \end{center}
    where $x \in X$ and $\epsilon \in G^{+}$. If $X$ is already a topological space, we say that the $G$-metric $d$ is compatible with the topology of $X$, if the latter coincides with the $d$-topology. A topological space is $G$-$\mathbf{metrizable}$ if it admits a compatible $G$-metric.

    Let $(X, d)$ be a $G$-metric space. A sequence $(x_{i})_{i < \kappa}$ of point from X is $d$-Cauchy if
    \begin{center}
        $\forall \epsilon \in G^{+} \exists \alpha < \kappa \forall \beta, \gamma \geq \alpha (d(x_{\beta}, x_{\gamma}) <_{G} \epsilon)$.
    \end{center}
    The space (X, $d$) is Cauchy-complete if every Cauchy sequence $(x_{i})_{i < \kappa}$ converges to some $x \in X$, that is,
    \begin{center}
        $\forall \epsilon \in G^{+} \exists \alpha < \kappa \forall \beta \geq \alpha (d(x_{\beta}, x) <_{G} \epsilon)$.
    \end{center}

    \begin{defin}\label{GenPolSp}
        A space $X$ is $\mathbf{G}$-$\mathbf{Polish}$ if it is completely $G$-metrizable and has weight $\leq \kappa$.
    \end{defin}

    Clearly, $G$-Polish Spaces are closed under closed subspaces. Moreover, the space $\kappa^{\kappa}$ (endowed with the bounded topology) is always $G$-Polish, as witnessed by the $G$-metric 
    \begin{center}
        d(x, y) = $ \left\{ 
        \begin{array}{ll}
        0_{G} & \text{if $x = y$}\\
        r_{\alpha} & \text{if $x|_{\alpha} = y|_{\alpha}$ and  $x(\alpha) \neq y(\alpha)$}
        \end{array} \right.$
    \end{center}
    where $(r_{\alpha})_{\alpha < \kappa}$ is a strictly decreasing sequence coinitial in $G^{+}$. It follows that all closed subspaces, including $2^{\kappa}$, are $G$-Polish for any $G$ as above.

    \begin{defin}
        Let $X$ be a space. A set $A\subseteq X$ is $G^{\kappa}_{\delta}$ if it can be written as a $\kappa$-sized intersection of open sets of $X$.
    \end{defin}

    \begin{thm}
        Let $X$ be a $G$-Polish space and $Y \subseteq X$. Then $Y$ is $G$-Polish if and only if $Y$ is $G^{\kappa}_{\delta}$ in $X$.
    \end{thm}

    \section{Defining a Topology on AEC's}\label{sectiontopologyAEC}

    In this section, we assume that $\mathbf{K}$ is an AEC with $\kappa = LS(\mathbf{K})$ in a vocabulary $\tau$, where $|\tau| \leq \kappa$. We define a topology on the models of cardinality $\kappa$ that belong to $\mathbf{K}$ and prove that they form a $G$-Polish space. 
    
    For the smooth development of the theory we isolate the following assumption which we assume throughout the rest of this paper.
        
    \vspace{6pt}    
    \textbf{Assumption:}      Let $\lambda \geq \kappa$. For $\mu= (\lambda + \kappa^{+})^{\kappa}$, it holds that $\mu^{<\mu}=\mu$. 
    \vspace{6pt}    
    
    The proof follows \cite{Hjorth}. We use the Relational Presentation Theorem and take as basic open sets those of the form:
    \begin{center}
        $\{ M | M \models \phi(\alpha_{1}, \alpha_{2}, ...)\}$
    \end{center}
    for all $\phi \in \mathcal{L}_{(2^{\kappa})^{+}, \kappa^{+}}$ and $\bar{\alpha} = (\alpha_{1}, \alpha_{2}, ...) \in \kappa$. Then we prove that our class is a $G^{\kappa}_{\delta}$ subset of the space $2^{2^{\kappa}}$ and thus a $G$-Polish space.

    We give some definitions first.

\begin{defin}\label{genfragment}
     Let $\lambda \geq \kappa$ be infinite cardinals and for $\mu= (\lambda + \kappa^{+})^{\kappa}$ it holds that $\mu^{<\mu}=\mu$. 
     
     We call a set of $\mathcal{L}_{\mu^{+}, \kappa^{+}}$-formulas, $\mathcal{F}$, a $\mathbf{fragment}$ if there is a set of variables $V$ of cardinality at least $\kappa^{+}$, such that if $\phi \in \mathcal{F}$, then all variables occurring in $\phi$ are in $V$ and $\mathcal{F}$ satisfies the following properties: 
        \begin{enumerate}
            \item all atomic formulas using only variables from $V$ and constant symbols are in $\mathcal{F}$
            \item if $\phi \in \mathcal{F}$ and $\psi$ is a subformula of $\phi$, then $\psi \in \mathcal{F}$
            \item if $\phi \in \mathcal{F}, v$ is free in $\phi$ and t is a term where every variable is in $V$, then the formula obtained by substituting t into all free occurrences of $v$ is in $\mathcal{F}$  
            \item $\mathcal{F}$ is closed under $\neg$
            \item $\mathcal{F}$ is closed under $\exists v$ for $v \in V$ (for finite length, i.e. it may not be closed for $\exists \bar{v}$, where $\bar{v}$ has infinite length)
            \item If $(\phi_{i})_{i < \sigma}$ are formulas in $\mathcal{F}$, then $\bigwedge_{i < \sigma}\phi_{i} \in \mathcal{F}$, for $\sigma < \mu$
            \item every formula in $\mathcal{F}$ has at most $\kappa$ many free variables
        \end{enumerate}
\end{defin}

    \underline{Observations:} \begin{enumerate} 
    \item If $T$ is an $\mathcal{L}_{\mu^{+}, \kappa^{+}}$-theory (it could be just a sentence), then there exists a minimum fragment that contains that theory. In order to construct that minimum fragment, one should include all the atomic formulas with variables in $V$, all the subformulas of $T$ and then include every other formula required by the closure properties $(3)$-$(7)$.
    \item For us, $\lambda$ will be the size of the models on which we define a topology. In this section, $\lambda$ is equal to $\kappa$ so $\mu = (\lambda + \kappa^{+})^{\kappa} = 2^{\kappa}$.
    \end{enumerate}

    We are now ready to define the topology:

    \begin{defin}\label{topology}
\begin{enumerate}
    \item $Mod_{\kappa}(\tau)$ is the space of $\tau$-structures whose  universe is $\kappa$.
    \item For $\mathcal{F} \subseteq \mathcal{L}_{\mu^{+}, \kappa^{+}}(\tau)$ a fragment, we let $T_{\mathcal{F}}$ be the $\mathbf{topology}$ on $Mod_\kappa(\tau)$ generated by the basic sets

        \begin{center}
            $\{ M | M \models \phi(\alpha_{1}, \alpha_{2}, ...)\}$,
        \end{center}

        where $\phi(\bar{x}) \in \mathcal{F}$ and $\bar{\alpha} = (\alpha_{1}, \alpha_{2}, ...) \in \kappa$. 
\end{enumerate}    
    \end{defin}

    \begin{thm}\label{THMtopology}
        Let $\mathbf{K}$ be an AEC in a vocabulary $\tau$. If $\mathcal{F}$ is the minimum fragment of $\mathcal{L}_{\mu^{+}, \kappa^{+}}(\tau^{*})$ that contains the theory $T^{*}$ from the Relational Presentation Theorem, then $(\mathbf{K}_{\kappa}, T_{\mathcal{F}})$ is a $G$-$Polish$ space. 
    
        To be more accurate, $T_{\mathcal{F}}$ cannot be defined on $\mathbf{K}_{\kappa}$, since the models in $\mathbf{K}$ are $\tau$-structures and $T_{\mathcal{F}}$ is defined on $\tau^{*}$-structures. But we will overlook this fact since there is a bijection between the models in $\mathbf{K}$ and their expansions in the Relational Presentation theorem.
    \end{thm}

    \begin{proof}
        First, we observe that $|T^{*}| = I(\mathbf{K}, \kappa) + \kappa \leq 2^\kappa=\mu$. Then, we need to bound the size of the fragment $\mathcal{F}$ generated by $T^*$. 
        
        For that purpose observe that since there are $\kappa^{+}$-many variables and at most $|\tau| \leq \kappa$-many constants, the atomic formulas are at most $(\kappa^{+})^{\kappa} = 2^{\kappa} = \mu$-many. Furthermore, the subformulas of sentences in $T^{*}$ and the formulas produced by substituting variables with terms are at most $2^{\kappa}$ by an easy computation. It is also easy to see that the rest of the closure properties of Definition \ref{genfragment} do not produce more than $\mu$-many new formulas. Notice here that the hypothesis $\mu^{< \mu} = \mu$ is needed when dealing with closure property (6). So, we finally have that $|\mathcal{F}| = \mu$.

        Let $C = \{ c_{i} | i < \kappa \}$ be a set of new constants and let $\hat{\tau}^{*} = \tau^{*} \cup C$, where $\tau^{*} \supseteq \tau$ is the extended vocabulary from the Relational Presentation Theorem. Define $\hat{\mathcal{F}}$ to be the fragment generated by $\mathcal{F}\cup C$ in $\mathcal{L}_{\mu^{+}, \kappa^{+}}(\hat{\tau}^{*})$ and define $S$ to be the set of all sentences in $\hat{\mathcal{F}}$. 
        
        First we observe that since we just added $\kappa$-many constant symbols, the computations are the same for $\hat{\mathcal{F}}$ as they were for $\mathcal{F}$. So $|\hat{\mathcal{F}}| = |\mathcal{F}| = \mu$. 
        
        We can also see that $|S| = \mu$ (the computations are again the same with the difference that we do not have $\kappa^{+}$-many variables,  but we have exactly $\kappa$-many constants, so the atomic formulas are $\kappa^{\kappa} = \mu$-many), 
        and therefore $2^{S} = \{ 0, 1 \}^{S}$ is a $G$-$Polish$ space.

        Let $B$ be the set of all functions $f \in 2^{S}$ satisfying the following properties:

        \begin{enumerate}
            \item any finite subset of $\{ \phi | f(\phi) = 1 \}$ is consistent
            \item for all $\phi$, we have that $f(\phi) = 0$ iff $f(\neg \phi) = 1$
            \item $f(\bigvee_{i < \mu} \phi_{i}) = 1$ iff there is some $i$ with $f(\phi_{i}) = 1$
            \item for all $\phi$, we have $f(\exists \bar{x} \phi(\bar{x})) = 1$ iff there is some $\bar{c} \in C$ with $f(\phi(\bar{c})) = 1$ (the length of $\bar{x}$ could be infinite).
        \end{enumerate}

        \underline{Claim:} $B$ is a $G_{\delta}^{\mu}$ subset of $2^{S}$.

        \underline{Proof of claim:} It suffices to show that the conditions (1)-(4) correspond to $G_{\delta}^{\mu}$ sets.

        (1) Let $[S]^{< \omega}$ be the set of all finite subsets of $S$ and $I \subseteq [S]^{< \omega}$ the set of all inconsistent finite subsets of $S$. 
        Then
        (1) corresponds to:

                \[\bigcap_{A \in I} \bigcup_{\psi \in A} \{ f | f(\psi) = 0 \},\]
        
        which is $G_{\delta}^{\mu}$, since with the topology we defined on the generalized Cantor space the sets $\{ f | f(\psi) = 0 \}$ are open and the size of the intersection is $|I| \leq \mu$. 

        \underline{Observation:} In fact the set $\{ f | f(\psi) = 0 \}$ is clopen, since its complement can be written as $\{ f | f(\psi) = 1 \}$, which is also open.

        (2) The second condition corresponds to the set:

            \[\bigcap_{\phi \in S} \{ f | f(\phi) = 0 \Leftrightarrow f(\neg \phi) = 1 \}.\]

        The sets $\{ f | f(\phi) = 0 \Leftrightarrow f(\neg \phi) = 1 \}$ are clopen, since for every $\phi$ the basic open sets $\{ f | f(\phi) = 0\}$ are clopen. Therefore, the desired set is $G_{\delta}^{\mu}$.

        (3) The third condition corresponds to the intersection of the following sets (one for each direction):

            \[\bigcap_{i < \mu} \{ f | f(\phi_{i}) = 1 \Rightarrow f(\bigvee_{i < \mu} \phi_{i}) = 1 \},\] 
            which is $G_{\delta}^{\mu}$ and\\
            \[\{ f | f(\bigvee_{i < \mu} \phi_{i}) = 0 \} \cup \bigcup_{i < \mu} \{ f | f(\phi_{i}) = 1 \},\] 
            which is open.

        (4) The fourth condition also corresponds to the intersection of two sets, one for each direction:

            \[\bigcap_{\phi \in S} \bigcap_{\bar{c} \in C} \{ f | f(\phi(\bar{c})) = 1 \Rightarrow f(\exists \bar{x} \phi(\bar{x})) = 1 \}, \]
            which is $G_{\delta}^{\mu}$ since the first intersection has size $\mu$ and the second intersection has size $\kappa$ and
        \[\bigcap_{\phi \in S} ( \{ f | f(\exists \bar{x} \phi(\bar{x})) = 0 \} \cup \bigcup_{\bar{c} \in C} \{ f | f(\phi(\bar{c})) = 1 \} ),\] which is also $G_{\delta}^{\mu}$.

        Thus, we have proven that $B$ is a $G$-Polish space.

        The next step is to prove that $Mod_{\kappa}(\tau^{*})$ is homeomorphic to $B$ and thus also a $G$-Polish space.
        We define the map:

        \begin{center}
            $e : Mod_{\kappa}(\tau^{*}) \to B$\\
            $M \mapsto T_{M}$,
        \end{center}
        as follows

        \begin{center}
            $T_{M}(\phi(c_{i_{1}}, c_{i_{2}}, ... )) = 1 \Leftrightarrow M \models \phi(i_{1}, i_{2}, ... )$
        \end{center}

        \underline{Claim:} $e$ is homeomorphism.

        \underline{1-1:} If $M_{1} \neq M_{2}$, there will be some atomic formula $\psi(x_{1}, x_{2}, ... )$ and $i_{1}, i_{2}, ... \in \kappa$ on which they disagree, whence $T_{M_{1}}$ and $T_{M_{2}}$ disagree on $\psi(c_{i_{1}}, c_{i_{2}}, ... )$.

        \underline{Onto:} Fix some $T \in B$. We define $M$ as follows:

        For any relation $R \in \tau^{*}$

        \begin{center}
            $M \models R(i_{1}, i_{2}, ...) \Leftrightarrow T(R(c_{i_{1}}, c_{i_{2}}, ...)) = 1$.
        \end{center}

        For any function $F \in \tau^{*}$

        \begin{center}
            $F^{M}(i_{1}, i_{2}, ...) = l \Leftrightarrow T(F(c_{i_{1}}, c_{i_{2}}, ...) = c_{l}) = 1$
        \end{center}
        (This is well defined. By (1) and (2) we have $T(\exists x (F(c_{i_{1}}, c_{i_{2}}, ...) = x)) = 1$ and then by (4) we have some witness $c_{l}$.)

        For any constant $c \in \tau^{*}$

        \begin{center}
            $c^{M} = i \Leftrightarrow T(c = c_{i}) = 1$
        \end{center} 
        (which is also well defined)

        Now we can easily prove 

        \begin{center}
            $M \models \phi(i_{1}, i_{2}, ...) \Leftrightarrow T(\phi(c_{i_{1}}, c_{i_{2}}, ...)) = 1$
        \end{center}
        by induction on $\phi \in \mathcal{F}$.

        \underline{$e$ is continuous:} If $G$ is is a basic subset of $B$, i.e. there is a $\lambda < \mu$ and $(\phi_{\alpha})_{\alpha < \lambda} \in \mathcal{F}$, such that the value $T(\phi_{\alpha})$ is the same for  every $T \in G$, we will prove that $e^{-1}(G)$ is an open set in $Mod(\tau^{*})$.
        
        $e^{-1}(G) = \{ M | M \models \bigwedge_{\alpha < \lambda} (\phi_{\alpha}^*) \}$, where $\phi_{\alpha}^*$ is either $\phi_{\alpha}$ or $\neg\phi_\alpha$ depending on whether the value of $T(\phi_\alpha)$ is $1$ or $0$ respectively. It follows that $e^{-1}(G)$ is an open set.

        \underline{Inverse function is continuous:} If $\{ M | M \models \phi(i_{1}, i_{2}, ...) \}$ is a basic set in $Mod_{\kappa}(\tau^{*})$, for some $\phi \in \mathcal{F}$ and $i_{1}, i_{2}, ... \in \kappa$,  then its image through $e$ is $\{ T \in B | T(\phi(c_{i_{1}}, c_{i_{2}}, ...) = 1 \}$, which is an open set in $B$. 
        
        This proves the claim.

        Finally, for $\mathbf{K}_\kappa$, we have:
        
            \[\mathbf{K}_{\kappa} = \bigcap_{\sigma \in T^{*}} \{ M|_{\tau} : M \in Mod_{\kappa}(\tau^{*}) \text{ and } M \models \sigma \}.\]
       
        This completes the proof, since \[\bigcap_{\sigma \in T^{*}} \{ M : M \in Mod_{\kappa}(\tau^{*}) \text{ and }M \models \sigma \}\] 
        is a $G_{\delta}^{\mu}$ subset of $Mod_{\kappa}(\tau^{*})$ and there is a bijection between the models of $\mathbf{K}$ and their expansions that model $T^{*}$.
    \end{proof}

\section{Generalization to Higher Cardinalities}\label{sectionGentopologyAEC}

In this section we observe that the main Theorem \ref{THMtopology} generalizes to higher cardinalities. In particular, we define a topology on models of $\mathbf{K}$ of size $\lambda > \kappa$  and prove that this is also a Generalized Polish Space. Throughout this section we assume that $\lambda > \kappa$ and that for $\mu = (\lambda + \kappa^{+})^{\kappa} = \lambda^{\kappa}$, it holds $\mu^{<\mu}=\mu$.
In addition, we  assume that the universe of all models of size $\lambda$, is $\lambda$. 

We now define the  topology:

    \begin{defin}\label{gentopology}
        For $\mathcal{F} \subseteq \mathcal{L}_{\mu^{+}, \kappa^{+}}(\tau)$ a fragment we let $T_{\mathcal{F}}$ be the $\mathbf{topology}$ on $Mod_{\lambda}(\tau)$, the space of $\tau$-structures on $\lambda$, with the following basic sets:

        \begin{center}
            $\{ M | M \models \phi(\alpha_{1}, \alpha_{2}, ...)\}$,
        \end{center}

        for all $\phi(\bar{x}) \in \mathcal{F}$ and $\bar{\alpha} = (\alpha_{1}, \alpha_{2}, ...) \in \lambda$. 
    \end{defin}

    \begin{thm}\label{genTHMtopology}
        Let $\mathbf{K}$ be an AEC in a vocabulary $\tau$. If $\mathcal{F}$ is the minimum fragment of $\mathcal{L}_{\mu^{+}, \kappa^{+}}(\tau^{*})$ that contains the theory $T^{*}$ from the Relational Presentation Theorem, then $(\mathbf{K}_{\lambda}, T_{\mathcal{F}})$ is a $G$-$Polish$ space.\footnote{The same comment  as in Theorem \ref{THMtopology} applies here too:   $T_{\mathcal{F}}$ cannot be defined on $\mathbf{K}_{\lambda}$, since the models in $\mathbf{K}$ are $\tau$-structures and $T_{\mathcal{F}}$ is defined on $\tau^{*}$-structures. But we can bypass this fact since there is a bijection between the models in $\mathbf{K}$ and their expansions in the relational presentation theorem.}
    \end{thm}

    \begin{proof}
The proof of this Theorem is similar to the proof of Theorem \ref{THMtopology}. We notice only a couple of differences between the two proofs.

Here we add to $\tau^*$ the set $C = \{ c_{i} | i < \lambda \}$ of $\lambda$-many new constants and define $\hat{\tau}^{*} = \tau^{*} \cup C$. We define $\mathcal{F}$, $\hat{\mathcal{F}}$ and $S$, as we did in Theorem \ref{THMtopology}. With similar computations as in the initial theorem, we can see that $|S| = \mu$ and that means that $2^{S} = \{0, 1\}^{S}$ is a $G$-Polish space.

The rest of the proof, if we replace $2^{\kappa}$ with $\lambda^{\kappa}$, is the same and we do not need to add anything new.
    \end{proof}

\section{Closing Remarks}\label{sClRem}

Since the models of a first order theory, equipped with the elementary substructure relation, form an AEC with countable L$\ddot{o}$wenheim-Skolem number, it comes natural to ask whether the results in this paper generalize the already known results presented in the introduction, namely that the countable models of a first order theory form a Polish Space. 

The answer to this question is negative. In the first-order case, we do not need to extend the vocabulary and appeal to a Presentation Theorem. The space of the countable models is homeomorphic to a subspace of $2^{\omega}$, which is a Polish Space. If we study the countable models of a first order theory in the terms of this paper, we first  extend the language and then we use the Relational Presentation Theorem. This results in  regarding countable models as models of an $\mathcal{L}_{\left(2^{\aleph_{0}}\right)^{+}, \aleph_{1}}$-theory and the space of countable models in the extended language is homeomorphic to a subspace of $2^{2^{\omega}}$, which is a $G$-Polish Space.

An open question is if we can derive similar results to this paper by using a Presentation Theorem different than the Baldwin-Boney Presentation Theorem. Since there is a plethora of other Presentation Theorems, it is interesting to ask if we can define a different topology on the models of an AEC which also gives rise to a $G$-Polish space.
 
Finally, this paper is a first attempt to generalize certain notions from first-order logic to AEC's. One natural question is if we can generalize notions like the atomic models, or define a topology on the models of an AEC so that the model-theoretic properties are closely connected with the topological properties, as it happens in the first-order case. We do not know the answer to this, but we see this paper as part of a larger project. 

\bibliographystyle{plain} 
\bibliography{bibliography.bib}

\begin{thebibliography}{1}

\bibitem{GPS}
Claudio Agostini, Luca~Motto Ros, and Philipp Schlicht.
\newblock Generalized polish spaces at regular uncountable cardinals, 2023.

\bibitem{Categoricity}
John~T. Baldwin.
\newblock {\em Categoricity}, volume~50 of {\em University Lecture Series}.
\newblock American Mathematical Society, Providence, RI, 2009.

\bibitem{PrTh}
John~T. Baldwin and Will Boney.
\newblock Hanf numbers and presentation theorems in {AEC}s.
\newblock In {\em Beyond first order model theory}, pages 327--352. CRC Press, Boca Raton, FL, 2017.

\bibitem{3GPS}
Samuel Coskey and Philipp Schlicht.
\newblock Generalized {C}hoquet spaces.
\newblock {\em Fund. Math.}, 232(3):227--248, 2016.

\bibitem{Baldwin}
Sy-David Friedman, Tapani Hyttinen, and Vadim Kulikov.
\newblock Generalized descriptive set theory and classification theory.
\newblock {\em Mem. Amer. Math. Soc.}, 230(1081):vi+80, 2014.

\bibitem{Hjorth}
Greg Hjorth.
\newblock Countable models and the theory of {B}orel equivalence relations.
\newblock In {\em The {N}otre {D}ame lectures}, volume~18 of {\em Lect. Notes Log.}, pages 1--43. Assoc. Symbol. Logic, Urbana, IL, 2005.

\bibitem{Grossberg}
Samson Leung.
\newblock Axiomatizing {AEC}s and applications.
\newblock {\em Ann. Pure Appl. Logic}, 174(5):Paper No. 103248, 18, 2023.

\bibitem{InfinitaryMarker}
David Marker.
\newblock {\em Lectures on infinitary model theory}, volume~46 of {\em Lecture Notes in Logic}.
\newblock Association for Symbolic Logic, Chicago, IL; Cambridge University Press, Cambridge, 2016.

\end{thebibliography}

\end{document}